\documentclass[12pt,twoside]{amsart}
\usepackage{amsmath}
\usepackage{amsthm}
\usepackage{amsfonts}
\usepackage{amssymb}
\usepackage{latexsym}
\usepackage{mathrsfs}
\usepackage{amsmath}
\usepackage{amsthm}
\usepackage{amsfonts}
\usepackage{amssymb}
\usepackage{latexsym}
\usepackage{geometry}
\usepackage{dsfont}
\usepackage[dvips]{graphicx}
\usepackage{color}
\usepackage[all]{xy}

\date{}
\allowdisplaybreaks[4] \footskip=15pt
\renewcommand{\uppercasenonmath}[1]{}

\usepackage{graphicx,amssymb}
\usepackage[all]{xy}
\usepackage{amsmath}

\allowdisplaybreaks
\usepackage{amsthm}
\usepackage{color}
\newtheorem*{Conflict of interest}{Conflict of interest}
\theoremstyle{plain}
\newtheorem{theorem}{Theorem}[section]
\newtheorem{proposition}[theorem]{Proposition}
\newtheorem{lemma}[theorem]{Lemma}
\newtheorem{corollary}[theorem]{Corollary}

\newtheorem*{open question}{Open Question}
\newtheorem{definition}[theorem]{Definition}
\newtheorem{question}{Question}

\theoremstyle{definition}
\newtheorem*{acknowledgement}{Acknowledgement}

\theoremstyle{remark}
\newtheorem{remark}[theorem]{Remark}

\def\p{\frak p}
\def\m{\frak m}

\def\Ker{{\rm Ker}}

\def\Im{{\rm Im}}
\def\Coker{{\rm Coker}}

\def\Mod{{\rm Mod}}

\begin{document}
\begin{center}
{\large  \bf Almost Noetherian rings and modules}

\vspace{0.5cm}   Xiaolei Zhang

{\footnotesize
\ School of Mathematics and Statistics,	Tianshui Normal University, tianshui 741001, China\\

E-mail: zxlrghj@163.com\\}
\end{center}

\bigskip
\centerline { \bf  Abstract}
\bigskip
\leftskip10truemm \rightskip10truemm \noindent

In this paper, we investigate the notions of almost Noetherian rings and modules.  In details, we give the Cohen type theorem, Eakin-Nagata type theorem, Kaplansky  type   Theorem and Hilbert basis theorem and some other rings constructions for almost Noetherian rings. In particular, we resolve a question proposed in \cite[B. Zavyalov,
{\it  Almost coherent modules and almost coherent sheaves},
Memoirs of the European Mathematical Society 19. Berlin: European Mathematical Society (EMS), 2025]{B25} under a certain condition.
\vbox to 0.3cm{}\\
{\it Key Words:} almost Noetherian ring,  almost Noetherian module, Cohen type theorem, Eakin-Nagata type theorem, Kaplansky  type   Theorem, Hilbert Basis Theorem.\\
{\it 2020 Mathematics Subject Classification:} 13E05.

\leftskip0truemm \rightskip0truemm
\bigskip

\section{Introduction}
Throughout this paper,  we fix a ``base'' commutative ring $R$ with an ideal $\m$ such that $\m^2= \m$. For a ring $R$,
we always do almost mathematics on $R$ with respect to $\m$. Let $M$ be an $R$-module, and $S$ a subset of $M$. We denote by $\langle S\rangle$ the $R$-submodule of $M$ generated by $S$.

The theory of almost rings was pioneered by G. Faltings in the late 1980s and 1990s as the essential machinery for his proofs of the major conjectures in $p$-adic Hodge Theory \cite{F88}. These results describe the deep structure of the cohomology of algebraic varieties over $p$-adic fields.  Faltings' insight was that by systematically neglecting torsion elements controlled by the maximal ideal, the proofs became dramatically more conceptual, transparent, and powerful.  The true testament to the power of almost ring theory came with the rise of perfectoid geometry, developed by P. Scholze \cite{S12}. The fundamental theorem of perfectoid geometry, the Tilting Correspondence, states that the geometry of a perfectoid algebra in characteristic zero is ``almost equivalent'' to the geometry of its tilt, a perfect algebra in characteristic $p$. This ``almost equivalence'' is expressed precisely through the language of almost isomorphisms. Almost ring theory provides the indispensable dictionary that translates problems from mixed characteristic to positive characteristic, where they are often dramatically simpler to solve. For more details on almost rings, please refer to \cite{GR03,B25}.

Noetherian rings, named after the groundbreaking mathematician E. Noether, is a cornerstone of commutative algebra and algebraic geometry. The Cohen type theorem states that a ring $R$ is a Noetherian if and only if every prime ideal of $R$ is finitely generated; the Eakin-Nagata type theorem states that a ring $R$ is  Noetherian if and only if a ring $T$ as its finitely generated module extension is also a Noetherian ring;  the Kaplansky  type  theorem states that  a ring $R$ is Noetherian if and only if it admits a faithful Noetherian module; the Hilbert basis theorem states that a ring $R$ is  Noetherian if and only if its polynomial ring $R[x]$ is a Noetherian ring. These results are fundamental and important in the area of commutative algebras. In the theory of almost mathematics, B.  Zavyalov \cite{B25} recently introduced the notion of almost Noetherian rings, which plays a key role in the almost ring theory. The main motivation of this paper is to extend the classical results in Noetherian rings as above to almost Noetherian rings. Moreover, we give some other rings constructions, such as trivial extensions, pull-backs and amalgamations, for almost Noetherian rings.

As our results concerns almost rings, we refer some basic notions from \cite{GR03,B25}.  An $R$-module $M$ is said to be almost zero, if $\m M$ is the zero module.
The category $\Sigma_R$, which is the full subcategory of $\Mod_{R}$ of all $R$-modules consisting of all  almost zero $R$-modules, is a Serre subcategory of $\Mod_{R}$. So, one can introduce the quotient category, which is called the category of \emph{almost $R$-modules},
\[
\Mod^a_{R}:=\Mod_{R}/\Sigma_R.
\]
Note that the localization functor
\[
(-)^a \colon \Mod_{R} \to \Mod^a_{R}
\]
is exact. We refer to elements of $\Mod^a_R$ as almost
$R$-modules or $R^a$-modules

A morphism $f\colon M\to N$ is called an \emph{almost isomorphism} (resp.\ \emph{almost injection}, resp.\ \emph{almost surjection}) if the corresponding morphism $f^a\colon M^a\to N^a$ is an isomorphism (resp.\ injection, resp.\ surjection) in $\mathrm{Mod}^a_R$. It follows by \cite[Lemma 2.1.8]{B25} that the morphism $f$ is an almost injection (resp.\ almost surjection, resp.\ almost isomorphism) if and only if $\operatorname{Ker}(f)$ (resp.\ $\operatorname{Coker}(f)$, resp.\ both $\operatorname{Ker}(f)$ and $\operatorname{Coker}(f)$) is an almost zero module.
\section{Almost Noetherian modules and their basic properties}

 Recall from \cite[Definition 2.5.1]{B25} that an $R$-module $M$ is said to be  \emph{almost finitely generated}, if for any $s \in \mathfrak m$ there are an integer $n_{s} \ge 0$ and an $R$-homomorphism
$f_{s} \colon R^{n_{s}} \to M$
such that $\operatorname{Coker}(f_{s})$ is killed by $s$, which is equivalent to that for any  $s \in \mathfrak m$ there exists a finitely generated submodule $N_s$ of $M$ such that $s M\subseteq N_s$. Recall from \cite[Definition 2.7.1]{B25} that a ring $R$ is said to be almost Noetherian if every ideal of $R$ is almost finitely generated.

To give a further study of  almost Noetherian rings, we introduce the notion of almost Noetherian modules.

\begin{definition}
	An almost finitely generated $R$-module $M$ is said to be almost Noetherian, if every submodule of $M$ is almost finitely generated.
\end{definition}

Trivially, a ring $R$ is an almost Noetherian ring if $R$ itself is an almost Noetherian $R$-module. Infinite direct sums of copies of almost zero non-zero modules is an almost zero module, and thus is almost Noetherian, but non-Noetherian.

\begin{remark}\label{pvr} \cite[section 2.11]{B25}
Fix a perfectoid valuation ring $K^+$ with perfectoid fraction field $K$, associated rank-$1$ valuation ring $\mathcal{O}_K=K^\circ$ , and ideal of topologically nilpotent elements $\m=K^{\circ\circ}\subset K^+$. Then $\m$ is flat over $K^+$ and $\widetilde{\m}\cong\m^2=\m$. It follows by \cite[Theorem 2.11.5]{B25} that any  a topologically finite type $K^+$-algebra is an almost Noetherian ring. Since every perfectoid valuation ring is not Noetherian, there exist non-trivial almost Noetherian modules which are not Noetherian.
\end{remark}


\begin{proposition} \label{s-u-noe-exact}
	Let  $0\rightarrow M\rightarrow N\rightarrow L\rightarrow 0$ be an exact sequence of $R$-modules. Then $N$ is almost  Noetherian if and only if $M$ and $L$ are almost  Noetherian.
\end{proposition}
\begin{proof} It is easy to verify that if $N$ is almost  Noetherian, then so are $M$ and $L$. Suppose that $M$ and $L$ are almost  Noetherian. Let $N'$ be a submodule of $N$.  Since $M$ is almost  Noetherian, for any  $s\in \m$ there exists finitely generated $R$-module $K_s$ such that  $s(M\cap N')\subseteq K_s\subseteq M\cap N'$. Since $L$ is almost  Noetherian, for any $t\in\m$ there exists  some finitely generated $R$-module $L_t$ such that $t[(N'+M)/M]\subseteq L_t\subseteq (N'+M)/M$. Let $N_{s,t}$ be the finitely generated submodule of $N'$ generated by the finite generators of $K_s$ and finite pre-images of generators of  $L_t$.  Consider the following natural commutative diagram with exact rows: $$\xymatrix@R=20pt@C=25pt{
		0 \ar[r]^{}&K\ar@{^{(}->}[d]\ar[r]&N_{s,t} \ar[r]\ar@{^{(}->}[d]&L_t\ar[r] \ar@{^{(}->}[d] &0\\
		0 \ar[r]^{}&M\cap N'\ar[r]&N' \ar[r]&(N'+M)/M \ar[r] &0.\\}$$
It is easy to check that  $stN'\subseteq N_{s,t}\subseteq N'$ for any $s,t\in \m$.
	Since $\m=\m^2$, every element $r$ in $\m$ can be written into $r=\sum\limits_{i}s_it_i$ for some finite $s_i,t_i\in \m$. $rN'\subseteq \sum\limits_{i}N_{s_i,t_i}\subseteq N'$. Note that $\sum\limits_{i}N_{s_i,t_i}$ is finitely generated. Hence, $N$ is almost  Noetherian.
\end{proof}

We call a sequence $\cdots\rightarrow M_{n+1}\xrightarrow{f_{n+1}} M_{n}\xrightarrow{f_{n}} M_{n-1}\xrightarrow{} \cdots$ of $R$-modules \emph{almost exact} at $M_n$ if 
 for any $s \in \mathfrak m$, $s\Ker(f_n)\subseteq\Im(f_{n+1})$ and $s\Im(f_{n+1})\subseteq\Ker(f_n)$. A sequence  of $R$-modules is called an \emph{almost exact sequence} if it is almost exact at each term. Certainly, an $R$-homomorphism $f:M\rightarrow N$ is an almost injection (resp.\ almost surjection, resp.\ almost isomorphism) if and only if  $0\rightarrow  M\xrightarrow{f} N$ (resp.\ $M\xrightarrow{f} N\rightarrow 0$, resp.\ $0\rightarrow  M\xrightarrow{f} N \rightarrow 0$) is almost exact.

\begin{theorem} \label{s-u-noe-s-exact}
	Let $0\rightarrow M\xrightarrow{f} N\xrightarrow{g} L\rightarrow 0$ be an almost exact sequence of $R$-modules. Then $N$ is almost  Noetherian if and only if $M$ and $L$ are almost  Noetherian
\end{theorem}
\begin{proof}  Let $0\rightarrow M\xrightarrow{f} N\xrightarrow{g} L\rightarrow 0$ be an almost exact sequence. Then for any  $s\in \m$ we have  $ s\Ker(g)\subseteq  \Im(f)$ and $ s\Im(f)\subseteq  \Ker(g)$. So $st\Ker(g)\subseteq s\Im(f)\subseteq \Ker(g)$ and $st\Im(f)\subseteq s\Ker(g)\subseteq \Im(f)$ for any $s,t\in\m$.  If $\Im(f)$ is almost  Noetherian, then the submodule $s\Im(f)$ of $\Im(f)$ is almost  Noetherian. Since $\m=\m^2$, every element $r$ in $\m$ can be written into $r=\sum\limits_{i}s_it_i$ for some finite $s_i,t_i\in \m$. So $r\Ker(g)\subseteq (\sum\limits_{i}s_i)\Im(f)\subseteq \Ker(g)$ for any $r\in\m$. Now, it is trivial to check that $\Ker(g)$ is almost  Noetherian. Similarly, if  $\Ker(g)$ is almost  Noetherian, then $\Im(f)$ is also almost  Noetherian. Consider the following three exact sequences:
	\begin{eqnarray*}
		0\rightarrow\Ker(g) \rightarrow  N\rightarrow \Im(g)\rightarrow 0, \\
		0\rightarrow\Im(g) \rightarrow  L\rightarrow \Coker(g)\rightarrow 0, \\
		0\rightarrow\Ker(f) \rightarrow  M\rightarrow \Im(f)\rightarrow 0
	\end{eqnarray*}
	with $\Ker(f)$ and $\Coker(g)$ almost  zero.   It is easy to verify that $N$ is almost  Noetherian if and only if $M$ and $L$ are almost  Noetherian by Proposition \ref{s-u-noe-exact}.
\end{proof}
\begin{corollary} \label{s-u-noe-u-iso}
	Let $M\xrightarrow{f} N$  an almost isomorphism of $R$-modules. If one of $M$ and $N$ is almost  Noetherian, then so is the other.
\end{corollary}

\begin{proof} This follows from Proposition \ref{s-u-noe-s-exact} since $0\rightarrow M\xrightarrow{f} N\rightarrow 0\rightarrow 0$ is an almost exact sequence.
\end{proof}

\begin{corollary}\label{exact-p}
Let $R$ be an almost Noetherian ring. Then
	$R^n$ is an almost Noetherian $R$-Noetherian $R$-module.
\end{corollary}
\begin{proof} Consider the exact sequence $0\rightarrow R^{n-1}\rightarrow R^n\rightarrow R\rightarrow0$. Following Proposition \ref{s-u-noe-exact}, it can be induced by induction on $n$.
\end{proof}

\begin{corollary}\label{afg-aNoe}
Let $R$ be an almost Noetherian ring. Then any almost finitely generated $R$-module is an almost Noetherian $R$-module.	
\end{corollary}
\begin{proof}  Suppose $M$ is an any almost finitely generated $R$-module, and $N$ an submodule of $M$. It follows by \cite[Lemma 2.5.15, Lemma 2.7.3]{B25} that $M/N$ is almost finitely presented. Then $N$ is almost finitely generated by   \cite[Lemma 2.5.15]{B25}.
\end{proof}

\section{Cohen type theorem for almost Noetherian rings and modules}

The well-known Cohen type theorem states that a ring $R$ is a Noetherian ring if and only if every prime ideal $\p$ of $R$ is finitely generated; and furthermore, an $R$-module $M$ is a Noetherian $R$-module if and only if every submodule of the form $\p M$  is finitely generated. In this section, we give the Cohen type theorem for almost Noetherian rings and almost Noetherian modules.

\begin{lemma}\label{prime}
	Let $R$ be a ring, $M$ be an almost finitely generated
	$R$-module. If $N$ is a submodule of $M$ which is maximal among all non-almost finitely generated
	submodules of $M$ $($that is, there is no non-almost finitely generated submodules which strictly contain $N)$, then $[N:M]:=\{r\in R\mid rM\subseteq N\}$ is a prime ideal of $R$.
\end{lemma}

\begin{proof}
	Set $\p=[N:M]$. Then $\p\not=R$ as $M$ is almost finitely generated but $N$ is not. On contrary, assume that $\p$ is not prime. Let $a,b\in R\setminus \p$ with $ab\in \p$. Then
	$N+aM$ is almost finitely generated as $N+aM$ strictly contains $N$. Hence,	for any $s\in \m$ there exist $n_{1,s},\dots,n_{p,s}\in N$ and $m_{1,s},\dots,m_{p,s}\in M$ such that
	\[
	s(N+aM)\subseteq\langle n_{1,s}+am_{1,s},\dots,n_{p,s}+am_{p,s}\rangle.
	\]
 Also, as $[N:a]:=\{m\in M\mid am\in N\}$ strictly contains $N$,
	$[N:a]$ is almost finitely generated. And so for any $t\in \m$, there exist $q_{1,t},\dots,q_{k,t}\in[N:a]$ such that
	\[
	t[N:a]\subseteq\langle  q_{1,t},\dots,q_{k,t}\rangle.
	\]
	 Now let $x\in N$. Then
	\[
	sx=\sum_{i=1}^{p} r_{i,s}(n_{i,s}+am_{i,s})
	\quad\text{for some } r_{i,s}\in R,
	\]
	so
	\[
	y=\sum_{i=1}^{p} r_{i,s}m_{i,s}\in[N:a].
	\]
	Then
	\[
	ty=\sum_{j=1}^{k} c_{j,t}q_{j,t}
	\quad\text{for some}\ c_{j,t}\in R.
	\]
	Therefore
	\[
	stx=\sum_{i=1}^{p} tr_{i,s}n_{i,s}+a\sum_{j=1}^{k} c_{j,t}q_{j,t}.
	\]
	So
	\[
	stN\subseteq\langle  tn_{1,s},\dots,tn_{p,s},aq_{1,t},\dots,aq_{k,t}\rangle\subseteq N.
	\]
Since $\m=\m^2$, every element $r$ in $\m$ can be written into $r=\sum\limits_{i}s_it_i$ for some finite $s_i,t_i\in \m$. Hence,	$$rN\subseteq\sum\limits_{i}\langle  t_in_{1,s_i},\dots,t_in_{p,s_i},aq_{1,t_i},\dots,aq_{k,t_i}\rangle\subseteq N.$$
Since the middle term is finitely generated, $N$ is almost finitely generated, which is a contradiction.
\end{proof}

\begin{lemma}\label{non-af-max}Let $R$ be a ring and $M$ be a non-almost Noetherian $R$-module. Then  the set $\mathcal{F}$
	of all non-almost finitely generated submodules of $M$ has a maximal element.
\end{lemma}
\begin{proof} Let $\{M_i\mid i \in\Gamma\}$ be an ascending chain of non-almost finitely generated submodules of $M$. Then we claim that  $\bigcup\limits_{i \in\Gamma} M_i$ is an upper bound. On contrary, if $\bigcup\limits_{i \in\Gamma} M_i$ is almost finitely generated, then for any $s\in\m$ there is a finitely generated submodule $M'_s$ of $\bigcup\limits_{i \in\Gamma} M_i$ such that $s\bigcup\limits_{i \in\Gamma} M_i\subseteq M'_s$. We can assume that $M'_s\subseteq M_{i_0}$ with $i_0\in\Gamma$. Then $s M_{i_0}\subseteq M'_s\subseteq M_{i_0}$ implying $M_{i_0}$ is almost finitely generated, which is a contradiction. So by Zorn's Lemma, one can choose an $R$-module $N$ maximal in $\mathcal{F}$.
\end{proof}

\begin{theorem}\label{cohentype} \textbf{$($Cohen type theorem for almost Noetherian  modules$)$}
	Let $R$ be a ring and $M$ be an almost finitely generated
	$R$-module. Then $M$ is almost Noetherian if and only if the submodules of the form
	$\p M$ are almost finitely generated for each prime ideal $\p$ of $R$.
\end{theorem}

\begin{proof}
	The ``only if'' part is clear. For the converse, assume that $\p M$ is
	almost finitely generated for each prime ideal $\p$ of $R$. On contrary, assume  that $M$ is not almost Noetherian. Set $\mathcal{F}$ to be the set
	of all non-almost finitely generated submodules of $M$. It follows by Lemma \ref{non-af-max} that one can choose an $R$-module $N$ whichis maximal in $\mathcal{F}$. Then by Lemma \ref{prime}, $\p:=[N:M]$ is a prime ideal.
	
As $M$ is almost finitely generated, for any $s\in\m$
there exists some finitely generated submodule $F_s$ of $M$ such that	$sM\subseteq F_s$ . If $\m\subseteq \p$,  then $tM\subseteq N$ for any $t\in\m\subseteq \p$. Hence, for any $s,t\in\m$ , we have $$tsN\subseteq tsM \subseteq tF_s \subseteq tM\subseteq N.$$Since $\m=\m^2$, every element $r$ in $\m$ can be written into $r=\sum\limits_{i}t_is_i$ for some finite $t_i,s_i\in \m$. Hence $$rN\subseteq \sum\limits_{i}t_iF_{s_i}\subseteq N.$$ Since the middle term is finitely generated, $N$ is almost finitely generated, which is a contradiction.

If $\m\not\subseteq \p$, then there exists ${t'}\in\m-\p$. Then we have
	\[
	\p =[N:M]\subseteq[N:F_{t'}]\subseteq[N:{t'}M]=[\p :{t'}]=\p,
	\]
	so $\p =[N:F_{t'}]$. Let $f_1,\dots,f_k$ generate $F_{t'}$. Then
	\[
	\p =[N:f_1]\cap\dots\cap[N:f_k],
	\]
	hence $\p =[N:f_i]$ for some $f_i$ which is denoted by $g$, because $\p$ is prime (see \cite[Proposition 1.11]{AM69}). Clearly $g\notin N$.
	By the maximality of $N$, $N+Rg$ is almost finitely generated,  so, for any $t\in \m$ there exist  $n_{1,t},\dots,n_{p,t}\in N$ and $a_{1,t},\dots,a_{p,t}\in R$ such that
	\[
	t(N+Rg)\subseteq\langle  n_{1,t}+a_{1,t}g,\dots,n_{p,t}+a_{p,t}g\rangle.
	\]
Set $N_t'=\langle  n_{1,t},\dots,n_{p,t}\rangle$.  For any $n\in N$, we have \[tn=\sum\limits_{i=1}^pr_i(n_{i,t}+a_{i,t}g)=\sum\limits_{i=1}^pr_in_{i,t}+\sum\limits_{i=1}^pr_ia_{i,t}g.\]
Hence  $\sum\limits_{i=1}^pr_ia_{i,t}\in[N:g]=\p$. Consequently,
	\[
	tN\subseteq N_t'+\p g\subseteq N_t'+\p M.
	\]
As $\p M$ is almost finitely generated by assumption, for any $v\in \m$ there exists some finitely generated submodule
	$G_v$ of $\p M$ such that
	$v\p M\subseteq G_v$. Then $$tvN\subseteq vN_t'+G_v\subseteq N$$ for any $t,v\in\m.$
	Since $\m=\m^2$, every element $r$ in $\m$ can be written into $r=\sum\limits_{i} t_iv_i$ for some finite $t_i,v_i\in \m$. Then  $$rN\subseteq \sum\limits_{i}( v_iN_{t_i}'+G_{v_i})\subseteq N.$$ Since the  middle term is finitely generated, $N$ is
	almost finitely generated,  a contradiction.
	
	In conclusion,  $M$ is an almost Noetherian $R$-module.
\end{proof}

\begin{corollary}\label{pNoer} \textbf{$($Cohen type theorem for almost Noetherian rings$)$}
	Let $R$ be a ring. Then $R$ is an almost Noetherian ring if and only if every prime ideal $\p$ of $R$ is almost finitely generated.
\end{corollary}
\begin{proof}
	Take $M=R$ in Theorem \ref{cohentype}.
\end{proof}

\section{Eakin-Nagata  theorem  for almost Noetherian rings}

In rest sections of this paper, we will investigate the almost Noetherian properties under change of rings. The following Lemma shows that we can do almost mathematics smoothly under the change of rings.

\begin{lemma}\label{rc}\cite[Lemma 2.1.11]{B25}
	Let $f\colon R\to S$ be a ring homomorphism, and let $\m_S$ be the ideal $\m S\subset S$. Then we have the equality $\m_S^2=\m_S$.
\end{lemma}
Let $f\colon R\to S$ be  a given ring homomorphism. We always do almost mathematics on $S$ with respect to $\m S$.

The well-known  Eakin-Nagata  theorem  states that if $R\subseteq T$ is an extension of rings with $T$ a finitely generated $R$-module, then  $R$ is a Noetherian ring if and only if so is $T$ (see \cite{E68,N68}). In this section, we give the Eakin-Nagata type theorem for almost Noetherian rings.

\begin{theorem}\label{en}\textbf{$($Eakin-Nagata type theorem for almost Noetherian rings$)$} Let $R$ be a ring, and $T$ a ring extension of $R$. If $T$ is almost finitely generated as an $R$-module. Then the following statements are equivalent.
	\begin{enumerate}
		\item $R$ is an almost Noetherian ring.
		\item $T$ is an almost Noetherian ring.
		\item  $\p T$ is an almost finitely generated $T$-ideal for every prime ideal $\p$ of $R$.
		\item $T$ is an almost Noetherian $R$-module.
	\end{enumerate}
\end{theorem}
\begin{proof} $(1)\Rightarrow (2)$ follows \cite[Lemma 2.8.3]{B25}. $(2)\Rightarrow (3)$ is obvious. $(3)\Rightarrow (4)$ follows from Theorem
\ref{cohentype} and  \cite[Lemma 2.8.3]{B25}.
	
	$(4)\Rightarrow (1)$ Suppose $T$ is an almost Noetherian $R$-module. Since $R$ is an $R$-submodule of $T$, $R$ is also a   almost Noetherian $R$-module by Theorem \ref{s-u-noe-s-exact}. It follows that $R$ is an almost Noetherian ring.
\end{proof}

\section{Kaplansky type  theorem for almost Noetherian rings}

Let $R$ be a ring and $M$ an $R$-module. Recall that $M$ is faithful if $[0:M]=0.$ The well-known Kaplansky  type  theorem states that  a ring $R$ is Noetherian if and only if it admits a faithful Noetherian $R$-module (see \cite[Exercise 2.32]{fk16}). We will give the Kaplansky theorem for almost Noetherian rings.
We say an $R$-module $M$  almost faithful if for any $s\in\m$ we have $s[0:M]=0$. Hence faithful $R$-modules are all almost faithful.

\begin{theorem}\label{conj} \textbf{$($Kaplansky  type   Theorem for almost Noetherian rings$)$} Let $R$ be a ring. Then $R$ is an almost Noetherian ring if and only if it admits an almost faithful almost Noetherian $R$-module.
\end{theorem}
\begin{proof} The necessity is trivial as $R$ itself is almost faithful almost Noetherian. For sufficiency, let $M$ be an almost Noetherian almost  faithful  $R$-module. Then $M$ is almost finitely generated, and so for any  $s\in \m$ there exist  $m_{1,s},\dots,m_{n,s}\in M$ such that $sM\subseteq \langle m_{1,s},\dots,m_{n,s}\rangle\subseteq M$. Consider the $R$-homomorphism $$\phi_s:R\rightarrow M^n$$ given by $$\phi_s(r)=(rm_{1,s},\dots,rm_{n,s}).$$ We claim that $ts\Ker(\phi_s)=0$ for any $t\in \m$. Indeed, let $r\in \Ker(\phi_s)$. Then $rm_{i,s}=0$ for each $i=1,\dots,n$. Hence $srM\subseteq r\langle m_{1,s},\dots,m_{n,s}\rangle=0$. And hence $sr\in [0:M]$. Since $M$ is an  almost faithful $R$-module, we have $tsr=0$ for any $t\in \m$, and so $ts\Ker(\phi_s)=0$. Note that $M^n$ is also an almost Noetherian $R$-module by continuously using Proposition \ref{s-u-noe-exact}, and so is its submodule $\Im(\phi_s)$.  Let $I$ be an ideal of $R$. Then $\phi_s(I)$ is a submodule of $\Im(\phi_s)$, and so is almost finitely generated. Thus for any $s'\in \m$ there exists $r_{1,s'},\cdots r_{m,s'}\in I$ such that
	$$s'\phi_s(I)\subseteq \phi_s(r_{1,s'}R+\cdots+r_{m,s'}R) \subseteq \phi_s(I).$$
	We claim that $tss'I\subseteq r_{1,s'}R+\cdots+r_{m,s'}R$. Indeed, for any $x\in I$, we have $s'\phi_s(x)=\phi_s(r_{1,s'}t_{1,s'}+\cdots+r_{m,s'}t_{m,s'})$ for some $t_{i,s'}\in R\ (i=1,\dots,m).$ Hence $\phi_s(r_{1,s'}t_{1,s'}+\cdots+r_{m,s'}t_{m,s'}-s'x)=0$. So
	$r_{1,s'}t_{1,s'}+\cdots+r_{m,s'}t_{m,s'}-s'x\in \Ker(\phi_s)$, and thus $ts(r_{1,s'}t_{1,s'}+\cdots+r_{m,s'}t_{m,s'})-tss'x=0$. It follows that  $$tss'I\subseteq ts (r_{1,s'}R+\cdots+r_{m,s'}R)\subseteq r_{1,s'}R+\cdots+r_{m,s'}R\subseteq I.$$
	Since $\m=\m^2$, we have $\m=\m^3$. Hence 	 every element $r$ in $\m$ can be written into $r=\sum\limits_{i}t_is_is'_i$ for some finite $t_i,s_i,s'_i\in \m$. Hence $$rI\subseteq \sum\limits_{i} r_{1,s'_i}R+\cdots+r_{m,s'_i}R\subseteq I.$$ Since the middle therm is finitely generated,  $I$ is almost finitely generated. So $R$ is an almost Noetherian ring.
\end{proof}
\begin{corollary}
	Let $R$ be a ring and $M$ be an $R$-module. If $M$ is an almost Noetherian $R$-module, then $R/[0:M]$ is an almost Noetherian ring.
\end{corollary}
\begin{proof} First we claim that $M$ is an almost Noetherian $R/[0:M]$-module.  Indeed, let $N$ be an $R/[0:M]$-submodule of $M$. Then it is also an $R$-submodule of $M$. Since $M$ is an almost Noetherian $R$-module, for any $s\in\m$ there exists a finitely generated submodule $N_s$ of $N$ such that $sN\subseteq N_s\subseteq N$. Note that  $N_s$ is also an $R/[0:M]$-submodule of $M$. So $(s+[0:M])N\subseteq N_s\subseteq N$, that is, $N$ is an almost finitely generated $R/[0:M]$-module. Consequently, $M$ is an almost Noetherian $R/[0:M]$-module. Since $M$ is also faithful as an $R/[0:M]$-module. It follows by Theorrem \ref{conj} that  $R/[0:M]$ is an almost Noetherian ring.	
\end{proof}

\section{Hilbert basis theorem for almost Noetherian rings}

It is well-known that any quotient ring of a  Noetherian ring is also a Noetherian ring.

\begin{proposition}\label{quot-Noether}
	Let $R$ be an almost Noetherian ring and $I$ an ideal of $R$. Then $R/I$ is also an almost Noetherian ring.
\end{proposition}
\begin{proof} Let $K:=J/I$ be an ideal of $R/I$ with $J$ an ideal of $R$ containing $I$. Then for any $s\in \m$, there is a finitely generated subideal $F_s$ of $J$ such that $sJ\subseteq F_s.$ Note that $\m R/I=(\m+I)/I$. Then every element in $\m R/I$ is of the form $s+I$ with $s\in\m$. Hence $$(s+I)K=(s+I)J/I\subseteq (F_s+I)/I\subseteq  K.$$
Since $(F_s+I)/I$ is a finitely generated ideal of $R/I$. Hence $R/I$ is also an almost Noetherian ring.
\end{proof}

The well-known Hilbert basis Theorem states that a ring $R$ is a Noetherian ring if and only if so is  $R[x]$. Hence a polynomial algebra in a finite number of variables over a Noetherian ring is also Noetherian. The author in \cite{B25} asked the following Question:
\begin{question} \cite[Warning 2.7.9]{B25}
If polynomial algebra in a finite number of variables over an almost Noetherian ring is also almost  Noetherian?
\end{question}
 The author \cite{B25} obtained that the above question is true for perfectoid valuation rings (see Remark \ref{pvr} or \cite[Theorem 2.11.5]{B25}).

Let $\m$ be an ideal of $R$. Let $n> 1$, denote by $\m^{[n]}=\{m^n\mid m\in\m\}$. We say an ideal $\m$ of $R$ \emph{Frobenius-powerful} provided that   $\m=\m^{[n]}$ for some $n>1$. Next we will give the  Hilbert basis theorem for almost Noetherian rings when $\m$ is Frobenius-powerful.

\begin{theorem}\label{poly-equ}\textbf{$($Hilbert basis theorem for almost Noetherian rings$)$}
	Let $R$ be a ring with $\m$  Frobenius-powerful.  Then $R$ is an almost Noetherian ring if and only if  $R[x]$ is an almost Noetherian ring.
\end{theorem}
\begin{proof} Suppose $R$ is an almost Noetherian ring. Let $I$ be an ideal of $R[X]$.  Set $J$ the ideal of $R$ consisting of the leading coefficients of polynomials in $I.$ Since $R$ is almost Noetherian, we have for any $s\in\m$ there exists  some  $a_{1,s},\dots,a_{n,s}\in J$ such that  $$sJ\subseteq\langle a_{1,s},\dots,a_{n,s}\rangle\subseteq J.$$ Choose $f_{i,s}\in I$ with leading coefficient $a_{i,s}$ and let $d_{i,s}$ be the degree of $f_{i,s}.$ Set $d_s= \max(d_{i,s})$. For any $f\in I$, write $f=ax^m+\cdots$. Then $a\in J$, and so  $sa\subseteq \langle a_{1,s},\dots,a_{n,s}\rangle$. Write  $sa=\sum\limits_{j}r_{j,s}a_{j,s}$ with some $r_{j,s}\in R$. If $m\geq d_s$, let $$g_s=sf-\sum\limits_{j}r_{j,s}x^{m-d_{j,s}}f_{j,s}.$$ Then $g_s\in I$ and  $\deg(g_s)<m$. If some $g_s$ has $\deg(g_s)\geq d_s$,
	continue this step. After finite steps,  we have $$s^{N(f)}f\in (I\cap F_s)+\langle f_{1,s},\dots,f_{n,s}\rangle$$ for some integer $N(f)>0$ where $F_s=R\oplus Rx\oplus\cdots\oplus Rx^{d_s-1}$. Since $\m$ is Frobenius-powerful,	 $\m=\m^{[n]}$ for some $n>1$, and hence $\m=\m^{[n^{N(f)}]}$. Note that  $s^{n^{N(f)}}f\in (I\cap F_s)+\langle f_{1,s},\dots,f_{n,s}\rangle$. Consequently, for any $s'\in\m$ there is $s\in\m$ such that $$s'f\in (I\cap F_{s})+\langle f_{1,{s}},\dots,f_{n,{s}}\rangle.$$

It follows by  Lemma \ref{exact-p} that $F_s$ is an almost Noetherian $R$-module. So $I\cap F_s$ is an almost finitely generated $R$-module. For any $t\in\m$, there exist $b_{1,t},\dots,b_{n_s,t}\in I\cap F_s$ such that $$t(I\cap F_s)\subseteq \langle b_{1,t},\dots,b_{n_s,t}\rangle\subseteq (I\cap F_s).$$  Set $$B_{s,t}=R[x]b_{1,t}+\cdots+ R[x]b_{n_s,t}.$$ Then for any $u\in I\cap F_s$,  $tu\in  \langle b_{1,t},\dots,b_{n_s,t}\rangle\subseteq B_{s,t}$.
	And so $t(I\cap F_s)\subseteq  B_{s,t}$.
	 Hence  $$s'tf\in  t(I\cap F_s)+t\langle f_{1,s},\dots,f_{n,s}\rangle\subseteq B_{s,t} +\langle f_{1,s},\dots,f_{n,s}\rangle\subseteq I$$ for any $s',t\in\m.$ Consequently, $$s'tI\subseteq B_{s,t}+\langle f_{1,s},\dots,f_{n,s}\rangle\subseteq I$$ for any $s',t\in\m$.
	 Since $\m=\m^2$, every element $r$ in $\m$ can be written into $r=\sum\limits_{i}s'_it_i$ for some finite $s'_i,t_i\in \m$. Then $rI\subseteq \sum\limits_{i}(B_{s_i,t_i}+\langle f_{1,s_i},\dots,f_{n,s_i}\rangle)\subseteq I$ of $R[x]$-ideals for some finite $s_i\in\m$. Then $rx^kI\subseteq\sum\limits_{i}(B_{s_i,t_i}+\langle f_{1,s_i},\dots,f_{n,s_i}\rangle)x^k\subseteq Ix^k\subseteq I$ for any $k\geq 0$ and any $r\in\m$.	Now, let $h(x)=\sum\limits_{k=0}^n r_kx^k\in\m[x]=\m R[x]$ with each  $r_k\in\m$. Then $r_kI\subseteq \sum\limits_{i}(B_{s_i,t_i,k}+\langle f_{1,s_i,k},\dots,f_{n_k,s_i,k}\rangle)\subseteq I$. Consequently, $$h(x)I\subseteq\sum\limits_{k=0}^n(\sum\limits_{i}(B_{s_i,t_i,k}+\langle f_{1,s_i,k},\dots,f_{n_k,s_i,k}\rangle))x^k\subseteq I,$$ implying $I$ is almost finite as the middle term is finitely generated. Hence $R[x]$ is an almost Noetherian ring.
	
	On the  other hand, suppose $R[x]$ is an almost Noetherian ring. Note that $R\cong R[x]/xR[x]$. It follows by Proposition \ref{quot-Noether} that $R$ is an almost Noetherian ring (Note we don't need  $\m$ to be Frobenius-powerful in this side).
\end{proof}

\begin{corollary}\label{poly-equ-1}
Suppose $R$ is an almost Noetherian ring  with $\m$  Frobenius-powerful. Then every finite type $R$-algebra is also an  almost Noetherian ring.
\end{corollary}
\begin{proof} Let $S$ be an finite type $R$-algebra. Then there exists an $n\geq 0$ and an surjection of  $R$-algebras $R[x_1,\dots,x_n]\twoheadrightarrow S$. It follows  by Proposition \ref{quot-Noether} and  Theorem \ref{poly-equ} that  $S$ is also an  almost Noetherian ring.
\end{proof}

\section{Ring constructions for almost Noetherian rings}

Let $R$ be a commutative ring and $M$ be an $R$-module. Then the \emph{trivial extension} of $R$ by $M$, denoted by $R\ltimes M$, is equal to $R\bigoplus M$ as $R$-modules with coordinate-wise addition and multiplication $$(r_1,m_1)(r_2,m_2)=(r_1r_2,r_1m_2+r_2m_1).$$ It is easy to verify that $R\ltimes
M$ is a commutative ring with identity $(1,0)$. Now we give an almost  Noetherian property on the trivial extension.

\begin{proposition}\label{trivial extension-usn} Let $R$ be a commutative ring, and $M$  an $R$-module. Then $R\ltimes
M$ is an almost Noetherian ring if and only if $R$ is an almost Noetherian ring and $M$ is an almost Noetherian $R$-module.
\end{proposition}

\begin{proof}  Suppose $R\ltimes
M$ is an almost Noetherian ring. Then it follows by Proposition \ref{quot-Noether} that $R$ is also an almost Noetherian ring as $R\ltimes
M/0\ltimes
M\cong R.$ Now, $0\ltimes
M$ is almost finitely generated. Then for any $(r,m)\in \m R\ltimes
M=\m\ltimes
\m M$, there exists a finitely generated subideal $$\langle (0,m_1),\dots,(0,m_n)\rangle\subseteq 0\ltimes
M$$ such that $$(r,m)0\ltimes
M\subseteq \langle (0,m_1),\dots,(0,m_n)\rangle.$$ Hence $rM\subseteq \langle m_1,\dots,m_n\rangle$. Consequently, $M$ is an almost Noetherian $R$-module.
	
Now suppose  $R$ is an almost Noetherian ring and $M$ is an almost Noetherian $R$-module.
Then $R\ltimes
M$ is almost finitely generated $R$-module. It follows by Theorem \ref{en} that  $R\ltimes
M$ is also an 	an almost Noetherian ring .	
\end{proof}

Let $\alpha: A\rightarrow C$ and $\beta: B\rightarrow C$  be ring homomorphisms.  Then the subring $$R:= \alpha \times_C \beta:= \{(a, b)\in  A\times B  \mid  \alpha(a) =\beta(b)\}$$ of $A\times B$ is called the \emph{pullback}  of $\alpha$ and $\beta$. Let $R$ be a pullback of $\alpha$ and $\beta$. Then there is a pullback diagram in the category of commutative rings:
$$\xymatrix@R=20pt@C=25pt{
	R\ar[d]^{p_B}\ar[r]^{p_A}& A \ar[d]^{\alpha}\\
	B\ar[r]^{\beta}&C. \\
}$$
 Now we give an almost Noetherian property on this type of pullback diagram.

\begin{proposition}\label{pullback-usn} Let $\alpha: A\rightarrow C$ be a ring homomorphism and $\beta: B\rightarrow C$  a  surjective ring homomorphism. Let $R$ be the pullback of $\alpha$ and $\beta$. Then
	the following conditions are equivalent:
	\begin{enumerate}
		\item  $R$ is an almost Noetherian ring;
		\item  $A$ is an almost Noetherian ring and $\Ker(\beta)$ is an almost Noetherian $R$-module.
	\end{enumerate}
\end{proposition}

\begin{proof} Let $R$ be the pullback of $\alpha$ and $\beta$. Since $\beta$ is a  surjective ring homomorphism, so is $p_A$.  Then there is a short exact sequence  of $R$-modules:
	$$0\rightarrow \Ker(\beta)\rightarrow R\rightarrow A\rightarrow 0.$$ By Proposition \ref{s-u-noe-s-exact}, $R$ is an almost Noetherian $R$-module if and only if $\Ker(\beta)$ and $A$ are almost Noetherian $R$-modules. Since $p_A$ is surjective, the $R$-submodules of $A$ are exactly the ideals of the ring $A$. Thus $A$ is an almost Noetherian $R$-module if and only if $A$ is an almost Noetherian ring.
\end{proof}

Let $f:A\rightarrow B$ be a ring homomorphism and $J$ an ideal of $B$. Following from \cite{df09} the  \emph{amalgamation} of $A$ with $B$ along $J$ with respect to $f$, denoted by $A\bowtie^fJ$, is defined as $$R=A\bowtie^fJ:=\{(a,f(a)+j) \mid a\in A,j\in J\},$$  which is  a subring of of $A \times B$.  By \cite[Proposition 4.2]{df09}, $A\bowtie^fJ$  is the pullback $\widehat{f}\times_{B/J}\pi$,
where $\pi:B\rightarrow B/J$ is the natural epimorphism and $\widehat{f}=\pi\circ f$:
$$\xymatrix@R=20pt@C=25pt{
	A\bowtie^fJ\ar[d]^{p_B}\ar@{->>}[r]_{p_A}& A\ar[d]^{\widehat{f}}\\
	B\ar@{->>}[r]^{\pi}&B/J. \\
}$$

Note that every ideal of $B$, for example $J$, can be viewed as an $A\bowtie^fJ$-module via  $p_B : A\bowtie^fJ\rightarrow B$ defined by $(a,f(a)+j)\mapsto f(a)+j)$.

\begin{proposition}\label{amag-usn}  Let $f :A\rightarrow B$ be a ring homomorphism, and $J$ an ideal of $B$. Then the following conditions
	are equivalent:
	\begin{enumerate}
		\item  $A\bowtie^fJ$ is an almost Noetherian ring;
		\item  $A$ is an almost Noetherian ring and $J$ is an almost Noetherian $A\bowtie^fJ$-module;
		\item  $A$ is an almost Noetherian ring and $f(A)+J$ is an almost Noetherian ring.
	\end{enumerate}
\end{proposition}

\begin{proof}  $(1)\Leftrightarrow(2)$ This follows from  Proposition \ref{pullback-usn}.
	
	$(1)\Rightarrow(3)$ By Proposition \ref{pullback-usn}, $A$ is an almost Noetherian ring. By \cite[Proposition 5.1]{df09}, there is a short exact sequence $$0\rightarrow f^{-1}(J)\times \{0\}\rightarrow A\bowtie^fJ \rightarrow f(A)+J\rightarrow 0$$ of $A\bowtie^fJ$-modules.   It follows by by Proposition \ref{quot-Noether} that  $f(A)+J$ is an almost Noetherian ring.

$(3)\Rightarrow(2)$ Let  $J_0$ be an $A\bowtie^fJ$-submodule  of $J$. Then
$J_0$ is an ideal of $f(A)+J$ as  every $A\bowtie^fJ$-submodule of $J$ can be seen as an ideal of $f(A)+J$.
	 Since $f(A)+J$ is an almost Noetherian ring, for any $$s+f^{-1}(J)\times \{0\}\in \m (f(A)+J)=(\m+f^{-1}(J)\times \{0\})/f^{-1}(J)\times \{0\}$$ with $s$ arbitrary in $\m$, there exist $j_1, \dots ,j_k\in J_0$ such that $$(s+f^{-1}(J)\times \{0\})J_0\subseteq \langle j_1, \dots ,j_k\rangle (f(A)+J)\subseteq J_0.$$ It is easy to check that
	$$sJ_0\subseteq \langle j_1, \dots ,j_k\rangle A\bowtie^fJ \subseteq J_0.$$
So $J_0$ is an almost finitely generated $A\bowtie^fJ$-module. Consequently,   $J$ is an almost Noetherian $A\bowtie^fJ$-module.
\end{proof}

\begin{acknowledgement} The author greatly appreciates the reviewer's comments and revision suggestions on this article.
\end{acknowledgement}

\begin{Conflict of interest}
The author states that there is no conflict of interest.	
\end{Conflict of interest}

\end{document}